\documentclass{amsart}
 \usepackage{graphicx,amssymb,amstext,amsmath}
\usepackage{adjustbox}
 \usepackage{enumerate}
 \usepackage{amsthm}
\usepackage{pdflscape}
\usepackage{mathtools}
\usepackage{rotating}

 \newtheorem{theorem}{Theorem}[section]
 \newtheorem{thm}[theorem]{Theorem}
 
 \newtheorem{lem}[theorem]{Lemma}
 
 \newtheorem{prop}[theorem]{Proposition}
 \newtheorem{claim}[theorem]{Claim}
 \newtheorem{conj}[theorem]{Conjecture}

\newcommand{\sym}{\mathop{\textrm{Sym}}}
\newcommand{\alt}{\mathop{\textrm{Alt}}}
\newcommand{\psu}{\mathop{\textrm{PSU}}}
\newcommand{\pgl}{\mathop{\textrm{PGL}}}
\newcommand{\psl}{\mathop{\textrm{PSL}}}

\newcommand{\Irr}{\mathop{\textrm{Irr}}}
\newcommand{\fix}{\mathop{\textrm{fix}}}
\newcommand{\one}{\bf{1}}
\newcommand{\notdivide}{\!\! \not | \,}

\renewcommand\arraystretch{1.25}

\begin{document}
\renewcommand\arraystretch{1.25}

\title[An Erd\H{o}s-Ko-Rado-type theorem for the group $\psu(3, q)$]{An
  Erd\H{o}s-Ko-Rado theorem for the group $\psu(3, q)$}

\author[K. Meagher]{Karen Meagher}

\thanks{Department of Mathematics and Statistics, University of Regina,
    Regina, Saskatchewan S4S 0A2, Canada. Research supported in part by an NSERC Discovery Research Grant,
    Application No.: RGPIN-341214-2013
       Email: \texttt{Karen.Meagher@uregina.ca}.}

\begin{abstract}
  In this paper we consider the derangement graph for the group
  $\psu(3,q)$ where $q$ is a prime power. We calculate all eigenvalues
  for this derangement graph and use these eigenvalues to prove that
  $\psu(3,q)$ has the Erd\H{o}s-Ko-Rado property and, provided that
  $q\neq 2, 5$, another property that we call the {\textsl Erd\H{o}s-Ko-Rado
  module property}.
\end{abstract}



\maketitle

\section{Introduction}

The Erd\H{o}s-Ko-Rado (EKR) theorem states that the largest set of
pairwise intersecting $k$-subsets from an $n$-set, when $n>2k$, has
size $\binom{n-1}{k-1}$ and consists of all subsets that contain a
fixed element of the $n$-set~\cite{ErKoRa}. This is a famous result
with many extensions and variations; there is a large body of work
showing that versions of the EKR theorem hold for many different
objects (see~\cite{EKRbook} and the references within). The focus of
this work is to show that a version of the EKR theorem holds for the
Projective Special Unitary group $\psu(3,q)$ where $q$ is a prime power.

Two permutations $\rho$ and $\sigma$ are said to be
\textsl{intersecting} if $\sigma^{-1}\rho$ has a fixed point (a
permutation with no fixed points is a \textsl{derangement}). If $\rho$
and $\sigma$ are intersecting, then $\rho$ and $\sigma$ both map some
point to a common point. Under this definition of intersection, one
can ask what is the size and the structure of the largest sets of
pair-wise intersecting permutations. Clearly the stabilizer of a point
(and any of its cosets) are intersecting sets of permutations under
this definition (indeed, these are the sets of all permutations that
map some fixed $i$ to some fixed $j$). We define the stabilizers of a
point, and all of the cosets of the stabilizer of a point, to be the
\textsl{canonical intersecting sets of permutations}. In~\cite{CaKu,
  GoMe, LaMa} it is shown that the largest sets of intersecting
permutations are exactly the canonical intersecting sets. This result
can be viewed as a version of the EKR theorem for permutations, since
it says that any largest set of intersecting permutations is the set
of permutations that map some fixed $i$ to some fixed $j$.

The same question may be asked for the permutations in a given group
(rather than all permutations). Define the canonical intersecting sets
for a permutation group to be the stabilizers of a point and their
cosets.  We say that a permutation group has the \textsl{EKR property}
if a canonical set is a largest set of pairwise interesting
permutations from the group.  Further, a permutation group has the
\textsl{strict-EKR property} if any pairwise intersecting set of
permutations from the group of maximum size is a canonical set. The
result in each of~\cite{CaKu, GoMe, LaMa} is that the symmetric group
has the strict-EKR property.  It has also been shown that many other
groups also have the strict-EKR
property~\cite{AhMeAlt,AhMetrans,MR2302532,KaPa1,KaPa2}.  It was
recently shown that every 2-transitive group has the EKR
property~\cite{MR3474795}.

In this paper we consider the group $\psu(3,q)$. The group $\psu(3,q)$
has a two-transitive action on a set of size $q^3+1$, we will only
consider this action. In \cite{MR3474795} it is shown that $\psu(3,q)$
has the EKR property, here we give more details of this result and
further give a result about the largest intersecting sets in
$\psu(3,q)$. Namely, we prove that the characteristic vector for any
maximum intersecting set in $\psu(3,q)$ is a linear combination of
characteristic vectors of the canonical cliques in $\psu(3,q)$.
Showing the same result for other 2-transitive groups was a key result
in proving that the groups have the strict-EKR
property~\cite{AhMeAlt,GoMe,KaPa1,KaPa2}.

\section{An algebraic proof of EKR theorems}

One effective method to show that a group has the EKR property is to
apply Hoffman's ratio bound (also known at Delsarte's bound) to the
\textsl{derangement graph} of the group. This method is used
in~\cite{AhMeAlt, GoMe, KaPa1, KaPa2, MR3474795}, and it will be the
focus of this paper.

The derangement graph of a group $G$, usually denoted by $\Gamma_G$,
has the elements of the group as its vertices, and two vertices are
adjacent if they are not intersecting. In particular, $\rho$ and
$\sigma$ from a permutation group $G$ are adjacent in $\Gamma_G$ if
and only if $\sigma^{-1}\rho$ is a derangement. With this definition,
a set of pairwise intersecting elements from $G$ is exactly a coclique
(or independent set) in $\Gamma_G$. For any group $G$, $\Gamma_G$ is a
Cayley graph with the set of derangements from the group as the
connection set.  Since the set of derangements is closed under
conjugation, $\Gamma_G$ is a \textsl{normal Cayley graph}.

It is well-known~\cite{Ba, MR626813} that it
is possible to calculate the eigenvalues for a normal Cayley graph
using the irreducible characters of the group.  Denote the
\textsl{irreducible complex characters} of a group $G$ by
$\Irr(G)$. Given $\psi\in\Irr(G)$ and a subset $S$ of $G$, we write
$\psi(S)$ for $\sum_{s\in S}\psi(s)$.  With these definitions, we can
give a formula for the eigenvalues the derangement graph for any
permutation group.

\begin{lem}\label{eigenvalues}
  Let $G$ be a permutation group and $D$ the set of
  derangements in $G$. The spectrum of the graph $\Gamma_G$ is
  $$\{\psi(D)/\psi(1)\mid \psi\in \Irr(G)\}.$$ Further, if $\lambda$ is an
  eigenvalue of $\Gamma_G$ and $\psi_1,\ldots,\psi_s$ are the
  irreducible characters of $G$ such that $\lambda=\psi_i(D)/\psi_i(1)$,
  then the dimension of the $\lambda$ eigenspace of $\Gamma_G$ is
  $\sum_{i=1}^s\psi_i(1)^2.$
\end{lem}

For each irreducible character of $G$ there is a corresponding
eigenvalue for $\Gamma_G$, and there is $(|G| \times |G|)$-matrix $E_\psi$ with
\[
[E_\psi]_{\rho, \sigma} = \psi(\sigma^{-1} \rho ).
\]
This matrix is a projection into the eigenspace for the eigenvalue
corresponding to $\psi$. The image of this projection is a $G$-module.
Throughout this paper, the group is fixed to be $\psu(3,q)$, so with an
abuse of notation, we called the $\psu(3,q)$-module that is the image
the projection $E_\psi$ the \textsl{$\psi$-module}.

As stated, the derangement graph for a group is a Cayley graph in
which the connection set is the union of the conjugacy classes of
derangements for the group. If $C$ is a conjugacy class in a group,
then we use $\Gamma_{C}$ to denote the Cayley graph on $G$ with
connection set $C$ (so $\rho$ and $\sigma$ are adjacent in $\Gamma_C$
if and only if $\rho \sigma^{-1} \in C$). Then the derangement graph
of $G$ is the union of all $\Gamma_{C}$ over all conjugacy classes of
derangements.  Since a character is constant on a conjugacy class, the
eigenvalues for $\Gamma_{C}$ are simply
\begin{eqnarray}\label{eq:calculateevalues}
\lambda_\psi(C) = \frac{|C|}{\psi(1)} \psi(c)
\end{eqnarray}
where $\psi$ is an irreducible character of $G$ and $c$ is any
element in the conjugacy class $C$. 
The eigenvalue of $\Gamma_G$ corresponding to $\psi$ is
the sum of $\lambda_\psi(C)$ over all conjugacy classes $C$ of
derangements.

If all of the eigenvalues of a derangement graph are known, Hoffman's
ratio bound~\cite[Section 2.4]{EKRbook} can be used to give an upper
bound on the size of the largest coclique in the graph; in some cases
this result can also be used to find a characterization of the
cocliques in the derangement graph. We will state Hoffman's bound only for
derangement graphs. 

For a set $S \subseteq G$, we use $v_S$ to denote the characteristic
vector of $S$; this is the length-$|G|$ vector with entry $1$ in
position $\sigma$ if $\sigma \in S$, and $0$ otherwise. For example
$v_G$ is the all ones vector, which we will denote by $\one$.

\begin{thm}\label{ratio}
  Let $G$ be a permutation group with $d$ derangements. Assume that
  $\tau$ is the minimum eigenvalue of $\Gamma_G$. Let $S$ be a
  coclique in $\Gamma_G$, then
\[
\frac{|S|}{|G|}\leq  \left(1-\frac{d}{\tau} \right)^{-1}.
\] 
If equality is met, then
  $v_S- \left( |S|/|G| \right) \one$ is an eigenvector of $\Gamma_G$ with
  eigenvalue $\tau$.
\end{thm}

It is noted in~\cite{AhMe} that for any 2-transitive group, the
character $\chi(g) = \fix(g) - 1$ is an irreducible character
(throughout this section we will use $\chi$ to denote this character).
Then, for any group with a 2-transitive action on a set of size $n$
that has exactly $D$ derangements, the eigenvalue of the derangement
graph corresponding to this irreducible character is $-|D|/(n-1)$
(see~\cite{AhMe} for details). For many groups (for example the
symmetric group, the alternating group and $\pgl(2,q)$), this is the
least eigenvalue of the derangement graph, and Hoffman's ratio bound
immediate shows that these groups have the EKR property (showing that
the groups have strict-EKR property is where the work lies). Further,
for both the symmetric and alternating group $\chi$ is the only
character that has $-|D|/(n-1)$ as its corresponding eigenvalue. In
these cases, Hoffman's bound implies that if $S$ is any maximum
coclique in $\Gamma_G$, then $v_S- (|S|/|G| ) \one$ is in the
$\chi$-module. This fact is used to characterize all maximum coclique
in $\Gamma_G$ in~\cite{AhMeAlt, GoMe}. Note that span of the
$\chi$-module and the all ones vector is the module corresponding to
the character $\fix(g)$, also known as the \textsl{permutation
  module}.

For a group $G$, define $S_{i,j}$ to be the set of all permutations in $G$ that map $i$
to $j$; if $i=j$, then $S_{i,j}$ is the stabilizer of a point,
otherwise it is the coset of the stabilizer of a point.  Define the
length-$G$ vectors $v_{i,j}$ to be the characteristic vectors for
$S_{i,j}$.  The vectors $v_{i,j}$ are the characteristic vectors of
the canonical cocliques. 
In~\cite{AhMe} it is shown that if $G$ is 2-transitive, then
\[
E_\chi v_{i,j} = \left( -\frac{|D|}{n-1} \right) (v_{i,j} - \frac{|G|}{|S|} \one)
\]
(this is a straight-forward calculation). Note that $v_{i,j} -
\frac{|G|}{|S|} \one$ is orthogonal to the all ones vector, thus it is
a \textsl{balanced vector}. From this we can conclude
that the vectors $v_{i,j} - \frac{|G|}{|S|} \one$ are in the $\chi$-module.
We summarize these results in the following lemma.

\begin{lem}[\cite{AhMe}]\label{instandard}
  Let $G$ be a $2$-transitive group and $S_{i,j}$ a canonical
  coclique in $G$. Then,
\begin{enumerate}
\item $v_{i,j} - \frac{1}{n} \one$ lies in  the $\chi$-module; and
\item $B:=\{v_{i,j}-\frac{1}{n}\mathbf{1}\,|\, i,j\in[n-1]\}$
is a basis for the $\chi$-module of $G$; and
\item $\{ v_{i,j}\,|\, i,j\in \{1,2,\dots,n\} \}$ is a spanning set
  for the permutation module.
\end{enumerate}
\end{lem}

If the characteristic vector for every maximum coclique in $\Gamma_G$
(where $G$ is a 2-transitive group) is in the $\chi$-module spanned by
these vectors, then we say that group has the \textsl{EKR-module}
property. The previous lemma implies that if $G$ has the EKR-module
property, then any maximum coclique in such a $\Gamma_G$ is a linear
combination of the characteristic vectors of the canonical
cocliques. The symmetric and alternating group has the EKR-module
property\cite{AhMeAlt, GoMe}, as does $\pgl(2,q)$, $\pgl(3,q)$ and the
Mathieu groups~\cite{AhMe,KaPa1, KaPa2}.

There are many 2-transitive groups for which the eigenvalue
corresponding to the character $\chi(g) = \fix(g) - 1$ is not the
least eigenvalue of the derangement graph. In~\cite{MR3474795} it is shown
that these groups do indeed have the EKR property; this is done using
a weighted version of Hoffman's bound.  For any graph a symmetric matrix $A$
with rows and columns indexed by the vertices of the graph is said to
be a \textsl{weighted adjacency matrix} of the graph if $A_{u,v}= 0$
whenever $u$ and $v$ are not adjacent vertices.  The following is a
weighted version of Hoffman's ratio-bound (see~\cite[Section
2.4]{EKRbook} for a proof) stated only for derangement graphs.

\begin{lem}\label{lem:weightedratio}
  Let $G$ be a permutation group and let $A$ be a weighted adjacency
  matrix of $\Gamma_G$. Let $d$ be the largest eigenvalue of $A$, and 
  $\tau$ the least eigenvalue of $A$. If $S$ is a coclique in
  $\Gamma_G$, then
\[
\frac{|S|}{|G|} \leq \left( 1 -\frac{d}{\tau} \right)^{-1}.
\]
\end{lem}

In this paper, we will calculate all the eigenvalues for
$\Gamma_{\psu(3,q)}$, we do this by calculating the eigenvalues for
some graphs $\Gamma_C$ where $C$ is the union of some of the conjugacy
classes of derangements in $\psu(3,q)$ (these graphs are defined
precisely in Sections~\ref{sec:evalueccone} and~\ref{sec:evaluecctwo}).
We will also calculate all the eigenvalues of the weighted
adjacency matrix for derangement graph $\Gamma_{\psu(3,q)}$ that is given
in~\cite{MR3474795}; this weighting assigns the weights to the conjugacy
classes of derangements. In particular, if $\rho$ and $\sigma $ are
adjacent, the weight in the weighted adjacency matrix will depend only
on which conjugacy class $\rho \sigma^{-1}$ belongs to.

The ratio bound on the weighted adjacency matrix shows that
$\psu(3,q)$ has the EKR property. This was done in~\cite{MR3474795}, but
the exact eigenvalues were not calculated. With the exact value of the
eigenvalues of the adjacency matrix we can prove that $\psu(3,q)$
has the EKR-module property. To do this we need the following new result.

\begin{lem}\label{lem:lotsareperm}
Assume that $G \leq \sym(n)$ is a 2-transitive permutation group with
$d$ derangements. If
\begin{enumerate}
\item there exists a weighted adjacency matrix $A$ with largest eigenvalue $k$ and least eigenvalue $\tau$ with
\[
\frac{1}{n}  = \left( 1 -\frac{k}{\tau} \right)^{-1};
\]
and \label{cond1}
\item $\chi(g) = \fix(g)-1$ is the only irreducible character of $G$ with eigenvalue equal to $-\frac{d}{n-1}$, \label{cond2}
\end{enumerate}
then $G$ has the EKR-module property.
\end{lem}
\begin{proof}
By Lemma~\ref{lem:weightedratio} and Condition~\ref{cond1}, the size of the largest coclique is $|G|/n$.

Let $S$ be any coclique of maximum size; so $|S| = |G|/n$. 
The quotient graph of $G$ with respect to the partition $\{S, G\backslash S\}$ is 
\[
Q = \left( 
\begin{matrix}
0 & d \\
\alpha & d-\alpha \\
\end{matrix} \right)
\]
where $d$ is the degree of the derangement graph for $G$ and
$\alpha = \frac{d |S| }{|G|-|S|} = \frac{d}{n-1}$ (this is found by
counting the edges between $S$ and $V\backslash S$ in $\Gamma_G$).
The eigenvalues of $Q$ are $d$ and $-\alpha$, these are also
eigenvalues of the adjacency matrix of the derangement graph on $G$
(corresponding to the trivial character and to $\chi$). The
eigenvalues of a quotient graph interlace the eigenvalues of the graph,
and if the interlacing is tight, the partition is
equitable~\cite[Lemma 9.6.1]{MR1829620}. Thus the partition
$\{S, G\backslash S\}$ is equitable. This means that each vertex in
$S$ is adjacent to exactly $d$ vertices in $G \backslash S$, and each
vertex in $G \backslash S$ is adjacent to exactly $\alpha$ vertices in
$S$ and $d-\alpha$ vertices in $G \backslash S$.

Let $v_S$ be the characteristic vector of $S$, then since $\{S, G\backslash V\}$ is equitable, 
\[
A \left( v_S -\frac{1}{n} {\one} \right) = -\frac{d}{n-1} \left( v_S - \frac{1}{n}\one\right). 
\]
Thus the balanced characteristic vector for any maximum coclique is a
$-(\frac{d}{n-1})$-eigenvector. Condition~\ref{cond2} this implies
that the $-(\frac{d}{n-1})$-eigenspace is exactly the $\chi$-module, so
$v_S -\frac{1}{n}\one$ is in the $\chi$-module (and that $v_S$ is in the
permutation module). Since $S$ is any coclique of maximum size, the
group $G$ has EKR-module property.
\end{proof}

The main result in this paper is that Lemma~\ref{lem:lotsareperm} can
be applied to $\psu(3,q)$. In Section~\ref{sec:evalueccone} and
Section~\ref{sec:evaluecctwo} all of the eigenvalues for graphs
$\Gamma_C$, where $C$ is a set of conjugacy classes of derangements in
$\psu(3,q)$, are calculated. Using these values, it is easy to find
all the eigenvalues of the derangement graph $\Gamma_{\psu(3,q)}$ for
all values of $q$, these are given in
Section~\ref{sec:evaluesAdj}. In this section
also includes a weighted adjacency matrix for $\psu(3,q)$, and Hoffman's ratio
bound holds with equality for this matrix.  From these results it will be
clear that Lemma~\ref{lem:lotsareperm} holds and $\psu(3,q)$ has
EKR-module property. We will consider the cases where $\gcd(3,q+1)=1$ and
$\gcd(3,q+1)=3$ separately.


\section{Eigenvalues for $\Gamma_{\psu(3,q)}$ with $\gcd(3,q+1) = 1$}
\label{sec:evalueccone}

If $3$ does not divide $q+1$, then the size of the group $\psu(3,q)$
is $(q^2-q+1)q^3(q+1)^2(q-1)$. There are two families of conjugacy
classes in $\psu(3,q)$ that are derangements, these families are
denoted by $C_1$ and $C_2$. The family $C_1$ consists $(q^2-q)/3$
conjugacy classes, and $C_2$ consists of $(q^2-q)/6$ conjugacy
classes. The rows of the character table for $\psu(3,q)$ corresponding
to these two families of conjugacy classes is given in the
appendix. The characters are grouped in families, labelled from
$\chi_1$ to $\chi_7$. The characters labelled $\chi_3,\chi_4$ and
$\chi_5$ represent a family of characters parametrized by the variable
given in the second row of the table; for the $\chi_6$ and $\chi_7$
the second row gives the number of these characters. The third row
gives the dimension of the character. In this table, the $(q+1)$-th
root of unity is denoted by $e$. Many details of this table are
omitted, see~\cite{MR0335618} for the complete character table of
$\psu(3,q)$.

The conjugacy classes of type $C_2$ are parameterized by triples from
the set
\[
T=\{ (k,l,m) \, : \, k+l+m \equiv 0 \pmod{q+1}, \, 1\leq k <l<m\leq q+1\}.
\]
The size of $T$ is $\frac{1}{q+1}\binom{q+1}{3} = \frac{q^2-q}{6}$.
The irreducible characters of type $\chi_5$ are also
parameterized by the set $T$; for these characters we will use the
triples $(u,v,w)$. The character $\chi_3$ in the table is the
character $\chi(g) = \fix(g) -1$.

In this section we calculate the eigenvalues of two graphs whose union
is the derangement graph of $\psu(3,q)$. The first is the union of all
$\Gamma_C$ where $C$ is a conjugacy class from the family $C_1$, we
denote this by $\Gamma_1$. The second is the union of all $\Gamma_C$,
where $C \in C_2$, this denoted by $\Gamma_2$. The eigenvalues for
$\Gamma_1$ and $\Gamma_2$ are, respectively,
\[
\lambda_\psi(\Gamma_1) = \sum_{C \in C_1} \lambda_\psi(C) = \sum_{C \in C_1} \frac{|C|}{\psi(1)} \psi(c),
\quad
\lambda_\psi(\Gamma_2) = \sum_{C \in C_2} \lambda_\psi(C) = \sum_{C \in C_2} \frac{|C|}{\psi(1)} \psi(c)
\]
(where $c \in C$). From Lemma~\ref{eigenvalues} it is clear that the
$\psi$-eigenvalue of $\Gamma_{\psu(3,q)}$ is simply the sum of the
$\psi$-eigenvaules of $\Gamma_1$ and $\Gamma_2$.

The first result is a statement of the eigenvalues that can be
calculated directly from the character table (Table~\ref{tab:char1}) given in the appendix
using Equation~\ref{eq:calculateevalues}.

\begin{lem}
Assume that $3 \notdivide q+1$, and $\Gamma_1$ and $\Gamma_2$ are as
defined above.
\begin{enumerate}
\item  The eigenvalues of $\Gamma_1$ and $\Gamma_2$ for the trivial
  character are $\frac{|G|(q^2-q)}{3(q^2-q+1)}$ and
  $\frac{|G|(q^2-q)}{ 6(q+1)^2}$ (respectively). 
\item For $\chi_1$ the eigenvalues of $\Gamma_1$ and $\Gamma_2$ are
  $-\frac{|G|}{3(q^2-q+1)}$ and
  $\frac{|G|}{3(q+1)^2}$ (respectively). 
\item The eigenvalues of $\Gamma_1$ and $\Gamma_2$ for $\chi_2$ are
  $-\frac{|G|(q-1)}{3q^2(q^2-q+1)}$ and
  $-\frac{|G|(q-1)}{6q^2(q+1)^2}$ (respectively). 
\item  The eigenvalue of $\Gamma_1$ for each of irreducible characters
  of type $\chi_3, \chi_4, \chi_5,\chi_6$ is 0.
\item The eigenvalue of $\Gamma_2$ for both $\chi_6$ and $\chi_7$ is
  0.
\end{enumerate}
\end{lem}

The first difficult calculation is for the eigenvalues of $\Gamma_1$
corresponding to the characters of type $\chi_3$. These
characters are parameterized by a variable $u \in \{1, \dots,
q+1\}$.  To calculate the sum of the value of $\chi_3$ on all the
elements in the conjugacy classes of type $C_2$, we need to determine
some facts about $T$.

\begin{claim}\label{Tshapeodd}
If $q$ is odd, then the following hold:
\begin{enumerate}
\item the element $q+1$ occurs in exactly $(q-1)/2$ triples in $T$;
\item any odd element from $\{1,\dots , q\}$ occurs in exactly $(q-1)/2$ triples in $T$;
\item any even element from $\{1,\dots , q\}$ occurs in exactly
  $(q-3)/2$ triples in $T$;
\end{enumerate}
\end{claim}
\proof

If $m=q+1$, then for each $k \in \{1,2,\dots, (q-1)/2\}$ there is
exactly one value for $\ell\in \{k+1,\dots, q\}$ such that $k+\ell =
q+1$. If $k > (q-1)/2$, then there are no such values of $\ell \leq
q$. Thus there are $(q-1)/2$ triples $(k,\ell,q+1) \in T$.

Assume $x \in \{1,\dots , q\}$ is odd and that $x+y+z \equiv 0
\pmod{q+1}$.  Since $q+1$ is even, $y$ and $z$ have different parities
and cannot be equal. For every odd value $y \in \{1,\dots ,q\}$,
except $y = x$, there is a unique $z \in \{1, \dots, q+1\}$ with $x+y+z \equiv 0
\pmod{q+1}$. Thus there are $q-1$ triples $x,y,z$ with $x+y+z \equiv 0
\pmod{q+1}$, but this counts each odd element twice. Thus there are
$(q-1)/2$ triples in $T$ that contain $x$.

Finally, consider when $x \in \{1,\dots , q\}$ is even. In this case
$y=z = (q+1 -x)/2 \pmod{q+1}$ is solution to $x+y+z \equiv 0
\pmod{q+1}$. This triple is not in $T$.  Thus there are $(q-1)/2-1=(q-3)/2$
triples in $T$ that contain $x$.
\qed

Using the same simple counting we get the parallel result for when $q$ is even.

\begin{claim}\label{Tshapeeven}
If $q$ is even, then the following hold:
\begin{enumerate}
\item the element $q+1$ occurs in exactly $q/2$ triples in $T$;
\item any element from $\{1,\dots , q\}$ occurs in exactly $(q-2)/2$ triples in $T$;
\end{enumerate}
\end{claim}

With these two lemmas we can now find the eigenvalues of $\Gamma_2$
for the characters of type $\chi_3$. The characters of type $\chi_3$
are parameterized by $u \in \{1,\dots,q+1\}$. For each of these
characters we calculate the sum of the character over all the
different conjugacy classes of type $C_2$ (which are parameterized by
the set $T$), and from this we find the value of the eigenvalue.

\begin{lem}
  If $q$ is odd, then the eigenvalue $\Gamma_2$ for the character $\chi_3$
  parameterized by $u = \frac{q+1}{2}$ is $- \frac{|G|(q-1)}{2(q+1)^2(q^2-q+1)}$.
\end{lem}
\proof 

Since $u = \frac{q+1}{2}$, if $i$ is even, then $e^{3ui} = 1$; and if
$i$ is odd, $e^{3ui} =-1$. From Lemma~\ref{Tshapeodd}, the element $0$
occurs in $(q-1)/2$ triples, any odd number occurs $(q-1)/2$ times
in a triple of $T$, and any even positive number occurs in
$(q-3)/2$ triples.  Thus, in the sum
\[
\sum_{(k,l,m) \in T} e^{3uk}+e^{3ul}+e^{3um}
\]
$1$ will occur 
\[
\frac{q-1}{2} + \left(  \frac{q-1}{2} \right)  \left( \frac{q-3}{2} \right) = \left(\frac{q-1}{2}\right)^2
\]
times.  While
$-1$ will occur 
\[
\left( \frac{q+1}{2} \right) \left( \frac{q-1}{2} \right)
\]
times.
The total value of this sum over all $(k,l,m) \in T$ is 
\[ 
\left( \frac{q-1}{2} \right)^2 -  \left( \frac{q+1}{2} \right)  \left(
  \frac{q-1}{2} \right) = -\frac{q-1}{2} .
\]
Thus the eigenvalue is
\[
- \frac{q-1}{2} \frac{|G|}{(q+1)^2}\frac{1}{q^2-q+1} = - \frac{|G|(q-1)}{2(q+1)^2(q^2-q+1)}.
\]
\qed

\begin{lem}
  For all values of $q$, the eigenvalue of the character $\chi_3$
  parameterized by any $u \neq \frac{q+1}{2}$ is
  $\frac{|G|}{(q^2-q+1)(q+1)^2} $.
\end{lem}
\proof
First we will prove that if $u \neq \frac{q+1}{2}$, then
\begin{align}\label{eq:chi3eqn}
\sum_{(k,l,m) \in T} e^{3uk}+e^{3ul}+e^{3um} =1.
\end{align}

If $q$ is even, then each $i \in \{1,\dots,q\}$ occurs exactly $\frac{q-2}{2}$
times in the sets of $T$, with $q+1$ occurring $\frac{q}{2}$ times.  Thus
\[
\sum_{(k,l,m) \in T} e^{3uk}+e^{3ul}+e^{3um} = e^{3u(q+1)} + \frac{q-2}{2} \sum_{i=1}^{q+1}e^{3ui} = 1
\]
(since $\sum_{i=1}^{q+1}e^{3ui} = 0$).

If $q$ is odd, then 
\[
\sum_{i \textrm{ even}}e^{3ui} =\sum_{i=1}^{(q+1)/2}(e^2)^{3ui}= 0,
\]
which implies that $\sum_{i \textrm{ odd}}e^{3ui}=0$. 
Then from Lemma~\ref{Tshapeodd}
\[
\sum_{(k,l,m) \in T} e^{3uk}+e^{3ul}+e^{3um} = 
e^{3u(q+1)} + \sum_{i \textrm{ odd}}e^{3ui}+ \frac{q-3}{2} \sum_{i=1}^{q+1}e^{3ui} = 1.
\]

From Equation~\ref{eq:chi3eqn}, the eigenvalue for this character
is
\[
(1) \frac{|G|}{(q+1)^2}\frac{1}{q^2-q+1} = \frac{|G|}{(q+1)^2(q^2-q+1)}.
\]
\qed

This result can also be used to find the eigenvalues of $\Gamma_2$
corresponding to the characters of type $\chi_4$.

\begin{lem}
  Assume that $q$ is odd.  The eigenvalue of $\Gamma_2$ for the
  character $\chi_4$ parameterized by $u = \frac{q+1}{2}$ is
  $\frac{|G|(q-1)}{2q(q+1)^2(q^2-q+1)}$. The eigenvalue of $\Gamma_2$
  for any other character of type $\chi_4$ is equal to
  $-\frac{|G|}{ q(q+1)^2(q^2-q+1)}$. \qed
\end{lem}

Each character of type $\chi_5$ is parameterized by triple a
$(u,v,w) \in T$.  The value of this character on one of the conjugacy
classes of type $C_2$ parameterized by $(k,\ell,m) \in T$, is $\sum_{[k,\ell,m]} e^{uk+v\ell+wm}$, where
the sum is taken over all permutations of $k,\ell,m$. Define $S$ to be
the set of all permutations of all the triples in $T$. Then
$\sum_{(k,\ell,m) \in S} e^{uk+v\ell+wm}$ is the sum of the value of
the character over every conjugacy class of type $C_2$. We next
calculate the value of this sum, but first we state a well-know result
(this is the generalization of the Chinese remainder theorem).

\begin{prop}\label{prop:numbersoln}
Let $d = \gcd(a,b,m)$.
The number of solutions to the equation
\[
ax+by =c \pmod{m}
\]
is $dm$ if $d$ divides $c$, and $0$ otherwise. \qed
\end{prop}

\begin{lem}\label{lem:chi5}
For $(u,v,w) \in T$, the value of
\[
  \sum_{(k,\ell,m) \in S} e^{uk+v\ell+wm}
\]
is equal to $-(q-1)$ if the set $(u,v,w)$ includes $q+1$, and
is equal to $2$ otherwise.
\end{lem}
\proof 
Set 
\[
d_u =\gcd(u, q+1),\quad d_{v} =\gcd(v, q+1), \quad d_{w}=\gcd(w, q+1).
\]
Further, set $d=\gcd(u,v,q+1)$ (since $u+v+w = q+1$, this implies that $d|w$).
Consider the number of solutions for $(k-m,\ell-m)$ in the equation
\begin{align}\label{eq:dio}
uk+v\ell+wm \equiv  u(k-m) + v (\ell-m) \equiv i \pmod{q+1}
\end{align}
where $i \in \{0, \dots, q\}$.

From Proposition~\ref{prop:numbersoln}, if $d$ does not divide $i$
then there are no solutions to the Equation~\ref{eq:dio}, and, if
$d|i$, then there are $d(q+1)$ solutions. Not all of these solutions
leads to a triple in $S$, since many of these solutions may have
either $k=\ell$, $k=m$ or $\ell=m$. We will count the number of solutions to
Equation~\ref{eq:dio} with  $k=\ell$, $k=m$ or $\ell=m$ and subtract these
from the total number of solutions. This will give the number of
triples in $S$ for which $uk+v\ell+wm \equiv i \pmod{q+1}$.

Assume that $k=m$. Equation~\ref{eq:dio} becomes
\[
v(\ell-m) \equiv i \pmod{q+1}.
\]
If $d_v$ divides $i$, then there are $d_v$ solutions for $\ell-m$ to
this equation. None these solutions corresponds to a triple in $S$
(since $k=m$).  Similarly, if $d_u$ divides $i$ there are $d_u$
solutions in which $\ell = m$. If $k=\ell$, then we can rewrite
Equation~\ref{eq:dio} as
\[
u(k-\ell) + w(m-\ell) \equiv  w (m-\ell) \equiv i \pmod{q+1}.
\]
If $d_w$ divides $i$, then this has exactly $d_w$ solutions, otherwise
there are no solutions.

We have shown that of the $d(q+1)$ solutions to Equation~\ref{eq:dio},
there are $d_v$ with $k=m$, $d_u$ with $\ell=m$ and $d_w$ with $k=\ell$.
None of these solutions correspond to a triple in $S$.  Next we
show that these solutions are all distinct.

Assume that there is a single solution for Equation~\ref{eq:dio} with
both $k=m$ and $k=\ell$.  Then, clearly, $k=\ell=m$ and
\[
ku+\ell v+m w = k(u+v+w) \equiv 0 \pmod{q+1}.
\]
This, with Equation~\ref{eq:dio} implies, that $i$ is congruent to $0$
modulo $q+1$. So, if $i$ not congruent to $0$ no solution can satisfy any
two of $k=\ell$, $k=m$ or $\ell=m$. Thus we can remove these solutions from
the total solutions with out removing one twice. On the other hand, if
$i$ is congruent to $0$ module  $q+1$, then there are solutions of the form
$(k,k,k)$. Each of these solutions will be counted $3$ times when we remove
them, rather than the one time that is correct.

If $i \not \equiv 0 \pmod{q+1}$ then the number of solutions $(k,\ell,m)$ to
Equation~\ref{eq:dio} with $(k,\ell,m) \in S$ is
\[
d(q+1) - \sum_{\stackrel{j \in \{u,v,w\}}{d_j | i}} d_j,
\]
and if $i \equiv 0 \pmod{q+1}$ then the number of solutions is
\[
d(q+1) - d_u -d_v -d_w +2.
\]

For an integer $i$ define $f(i)$ to be the set of $j \in \{u,v,w\}$
for which $d_j$ divides $i$. If $d$ is a divisor of $q+1$, other than
$q+1$, then $\sum_{\stackrel{i\in\{1,\dots,q+1\} }{d | i}}e^{i} =0$. If $d =q+1$, then
this sum would simply be $\sum_{\stackrel{i\in\{1,\dots,q+1\}}{q+1 | i}}e^{i} =e^{q+1} = 1$.
Thus we have that
\begin{align*}
 \sum_{(k,\ell,m) \in S} e^{uk+v\ell+wm} 
 &= d(q+1) \!\! \sum_{i\in\{1,\dots,q+1\}}  e^{i} +2 e^0 - 
\left(     (d_u+d_v+d_w) \!\! \sum_{\stackrel{i\in\{1,\dots,q+1\}}{f(i) =
      \{d_u,d_v,d_w\}}}\!\! e^{i}  \right.  \\
&+    (d_u+d_v)\!\!\!\! \sum_{\stackrel{i\in\{1,\dots,q+1\}}{f(i) = \{d_u,d_v\}}} \!\!e^{i} 
+    (d_u+d_w)\!\!\!\! \sum_{\stackrel{i\in\{1,\dots,q+1\}}{f(i) = \{d_u,d_w\}}}\!\! e^{i} 
+    (d_v+d_w)\!\!\!\! \sum_{\stackrel{i\in\{1,\dots,q+1\}}{f(i) = \{d_v,d_w\}}}\!\! e^{i}   \\
&\left. +    d_u\!\!\sum_{\stackrel{i\in\{1,\dots,q+1\}}{f(i) = \{d_u\}}} \!\!e^{i} 
+    d_v\!\!\sum_{\stackrel{i\in\{1,\dots,q+1\}}{f(i) = \{d_v\}}} \!\!e^{i} 
+    d_w\!\!\sum_{\stackrel{i\in\{1,\dots,q+1\}}{f(i) = \{d_w\}}} \!\!
e^{i}  \right) \\
&=2 - d_u \sum_{\stackrel{i\in\{1,\dots,q+1\}}{d_u | i}}e^{i} 
    - d_v \sum_{\stackrel{i\in\{1,\dots,q+1\}}{d_v | i}}e^{i}
    - d_w \sum_{{\stackrel{i\in\{1,\dots,q+1\}}{d_w | i}}}e^{i}.
\end{align*}
If the triple $\{u,v,w\}$ includes $q+1$, then this is equal to
$2-(q+1) = -(q-1)$, otherwise it is equal to $2$.  \qed

\begin{lem}
  The eigenvalue of $\Gamma_2$ corresponding to the character of
  type $\chi_5$ which are parameterized by a triple that contains
  $q+1$ is $- \frac{|G|}{(q+1)^2(q^2-q+1)}$.  The eigenvalue of
  $\Gamma_2$ for any other character of type $\chi_5$ is $
  \frac{-2|G|}{(q+1)^2(q-1) (q^2-q+1)}$.
\end{lem}
\proof Consider a character of type $\chi_5$ that is
parameterized by a triple that contains $q+1$.  Then the corresponding
eigenvalue of $\Gamma_2$ is
\[
(q-1)  \left(  \frac{|G|}{(q+1)^2} \right) \frac{1}{(q-1)(q^2-q+1)}   
= - \frac{|G|}{(q+1)^2(q^2-q+1)}.
\]

If $q$ is odd, then there are $(q-1)/2$ triples in $T$ that contain
$q+1$, while if $q$ is even, then there are $q/2$ triples in $T$ that
contain $q+1$. Thus this is an eigenvalue for $(q-1)/2$
characters when $q$ odd, and $q/2$ characters if $q$ is even.

Next consider the characters of type $\chi_5$ which are parameterized by a triple that
does not include $q+1$.  By Lemma~\ref{lem:chi5}, the eigenvalue for
these characters is
\[
(-2)  
      \left( \frac{|G|}{(q+1)^2} \right)     
      \frac{1}{(q-1)(q^2-q+1)}                
= \frac{-2|G|}{(q+1)^2(q-1) (q^2-q+1)}.   
\]

If $q$ is odd then there are $q(q-1)/6 -(q-1)/2 = (q-1)(q-3)/6$
triples in $T$ that do not contain $q+1$. Further, if $q$ is even then
there are $q(q-1)/6 - q/2 = (q^2-4q)/6$ triples in $T$ that do not
contain $q+1$. This gives the number of characters for which
$ \frac{-2|G|}{(q+1)^2(q-1) (q^2-q+1)}$ is the eigenvalue.\qed

\begin{lem}
The eigenvalue of $\Gamma_1$ for the character $\chi_7$ is
$\frac{|G|}{(q^2-q+1)(q+1)^2(q-1)}$ . 
\end{lem}
\proof
It can be seen from the full character table of $\psu(3,q)$ that the
sum of the $B$ for $\chi_7$ over the different conjugacy classes of
type $C_1$ is
$1$.

Next simply calculate
\[
 \frac{|G|}{(q^2-q+1)} \frac{1}{(q+1)^2(q-1)} 
=\frac{|G|}{(q^2-q+1)(q+1)^2(q-1)} .
\]
\qed

\begin{table}
\begin{tabular}{| l | c | c | c |} \hline
Character& Number & eigenvalue & eigenvalue\\ 
  &   &  of $\Gamma_1$ & of $\Gamma_2$ \\ \hline
 Trivial& 1 & $\frac{|G|(q^2-q)}{3(q^2-q+1)}$ & $\frac{|G|(q^2-q)}{6(q+1)^2}$ \\
 $\chi_1$ & 1 & $-\frac{|G|}{3(q^2-q+1)}$ &  $\frac{|G|}{3(q+1)^2}$ \\ 
$\chi_2$ & 1 &  $-\frac{|G|(q-1)}{3q^2(q^2-q+1)}$ &  $-\frac{|G|(q-1)}{6q^2(q+1)^2}$ \\
$\chi_3$, $u = \frac{q+1}{2}$ & $1$ & $0$ & $-\frac{|G|(q-1)}{2(q+1)^2(q^2-q+1)}$\\
$\chi_3$, $u \neq  \frac{q+1}{2}$ & $q$ or $q+1$ & $0$ & $ \frac{|G|}{(q^2-q+1)(q+1)^2} $\\
$\chi_4$, $u = \frac{q+1}{2}$ & $1$ & $0$ &$\frac {|G|  (q-1)}{ 2q (q^2-q+1)(q+1)^2}$\\
$\chi_4$, $u \neq \frac{q+1}{2}$ & $q$ or $q+1$ & $0$ &$-\frac{|G|}{q(q^2-q+1)(q+1)^2}$\\
$\chi_5$ & $\lfloor \frac{q}{2} \rfloor$ & $0$ & $-\frac{|G|}{(q+1)^2(q^2-q+1)}$ \\
$\chi_5$ & $\frac{q^2-q}{6}-\lfloor \frac{q}{2} \rfloor$ & $0$ & $ -\frac{2|G|}{(q+1)^2(q-1) (q^2-q+1)}$\\
$\chi_6$ & $\frac{q^2-q-2}{2}$ & $0$ &  $0$ \\
$\chi_7$ & $\frac{q^2-q}{6}$ & $\frac{|G|}{(q^2-q+1)(q+1)^2(q-1)}$& $0$ \\  \hline
\end{tabular}
\caption{Eigenvalues for the two types of conjugacy classes of
  derangements in $\psu(3,q)$ where $3\notdivide q+1$.~\label{tab:gcd1}}
\end{table}

\section{Eigenvalues for $\Gamma_{\psu(3,q)}$ with $\gcd(3,q+1) = 3$}
\label{sec:evaluecctwo}

In this section we do the same calculations as in the pervious
section, but for $\psu(3,q)$ with $\gcd(3,q+1) = 3$. In this case, the
size of the group $\psu(3,q)$ is
$\frac{1}{3}(q^2-q+1)q^3(q+1)^2(q-1)$. For these values of $q$, there
are three families of conjugacy classes of derangements in
$\psu(3,q)$. We denote these families by $C_1$, $C_2$ and $C_3$.
There are $(q^2-q-2)/9$ conjugacy classes of type $C_1$ and only one
of type $C_3$.  The conjugacy classes of types $C_2$ are parameterized
by the set $T$ of all triples $(k,\ell,m)$, with $1\leq k < \ell \leq
(q+1)/3$ and $\ell < m \leq q+1$ and $k+\ell+m = q+1$.  The number of
such pairs is
\[
\binom{\frac{q+1}{3}}{2}  = \frac{q^2-q-2}{18}.
\]
The characters of type $\chi_6$ are also parametrized by the
elements of $T$. Again we use $\Gamma_i$ to denote the Cayley graph on
the group $\psu(3,q)$ with connection set of all the conjugacy classes
of type $C_i$, for $i=1,2,3$. Again, $\Gamma_{\psu(3,q)}$ is the union
of these three graphs.

The rows of the character table for $\psu(3,q)$ corresponding to these
families of conjugacy classes is given in the appendix. In this table,
the $(q+1)$-th root of unity is denoted by $e$, and the third root of unity
is denoted by $\omega$. See~\cite{MR0335618} for the complete
character table of $\psu(3,q)$.

The first result is a statement of the eigenvalues that can be
calculated directly from the character table (Table~\ref{tab:char2}) given in the appendix
using Equation~\ref{eq:calculateevalues}.

\begin{lem}
\begin{enumerate}
\item The eigenvalues for the trivial character of $\Gamma_1$, $\Gamma_2$ and
$\Gamma_3$ are
$\frac{|G|(q^2-q-2)}{3(q^2-q+1)}$,
$\frac{|G|(q-2)}{6(q+1)}$, and
$\frac{|G|}{(q+1)^2}$  (respectively). 

\item The eigenvalues for $\chi_1$ of $\Gamma_1$, $\Gamma_2$ and
$\Gamma_3$ are
$-\frac{|G|(q^2-q-2)}{3q(q-1)(q^2-q+1)}$,
$\frac{|G|(q-2)}{3q(q-1)(q+1)}$, and
$\frac{2|G|}{q(q-1)(q+1)^2}$  (respectively). 

\item The eigenvalues for $\chi_2$ of $\Gamma_1$, $\Gamma_2$ and
$\Gamma_3$ are 
$-\frac{|G|(q^2-q-2)}{3q^3(q^2-q+1)}$,
$-\frac{|G|(q-2)}{6q^3(q+1)}$, and
$-\frac{|G|}{q^3(q+1)^2}$  (respectively).

\item The eigenvalue of $\Gamma_1$ for $\chi_3$, $\chi_4$, $\chi_5$,
$\chi_6$ and $\chi_7$ is 0.

\item The eigenvalue of $\Gamma_3$ is $\frac{3|G|}{(q+1)^2(q^2-q+1)}$
  for $\chi_3$ and $-\frac{3|G|}{q(q+1)^2(q^2-q+1)}$ for $\chi_4$. The
  eigenvalue for $\Gamma_3$ corresponding to $\chi_5$ is
  $\frac{-6|G|}{(q+1)^2(q-1)(q^2-q+1)}$ if $9|q+1$, and
  $\frac{3|G|}{(q+1)^2(q-1)(q^2-q+1)}$ otherwise. For $\chi_7$ and
  $\chi_8$ the eigenvalue of $\Gamma_3$ is $0$.
\end{enumerate}
\end{lem}

The next irreducible character that we consider is $\chi_3$. As in the
case where $3 \notdivide q+1$, we will find the eigenvalue
corresponding to $\chi_3$ by summing the value of the character over
all the conjugacy classes of type $C_2$.  First we need a technical
result.

\begin{lem}
For $k \in \{1,2, \dots, (q+1)/3-1\}$ the sum
\[
\sum_{(k,\ell,m) \in T} e^{3ku}+e^{3kv}+e^{3kw}
\]
(where $e$ is the $(q+1)$-th root of unity) is equal to $0$.
\end{lem}
\proof 
For any triple $(k,\ell,m) \in T$, both $k$ and $\ell$ are no more
than $(q+1)/3$ and $m$ is between $\ell$ and $q+1$.
Every distinct pair $(k,\ell)$ from $\{1, \dots, (q+1)/3\}$ will
be in exactly one triple $(k,\ell,m)$ in the set $T$. Thus each element
less than $(q+1)/3$ will occur in exactly $(q+1)/3 -1 = (q-2)/3$
triples in $T$.

Next we count the number of time an elements larger than
$(q+1)/3$ (these are the elements we represent with $m$) occur in a
triple in $T$.  First, consider
the elements $\frac{q+1}{3} +i$ where $i \in \{1, \dots,
\frac{q-2}{3}\}$.  Each of these elements will be in exactly $\lceil
\frac{i}{2} \rceil$ triples from $T$; these are the triples of the
form
\[
(\frac{q+1}{3}-i, \frac{q+1}{3}, \frac{q+1}{3} +i), \quad 
(\frac{q+1}{3}-i+1, \frac{q+1}{3}-1,\frac{q+1}{3} +i), \dots
\]
Similarly the elements $q-i$ with $i \in \{1, \dots, \frac{q-2}{3}\}$
will be in $\lfloor \frac{i}{2} \rfloor$ triples from $T$.

Consider elements $a$ and $b$ with $(q+1)/3 <a  < 2(q+1)/3 \leq b <(q+1)$ and
\[
3a \equiv 3b \pmod{q+1}.
\]
These conditions imply that $a = (q+1)/3 +i$ and $b= 2(q+1)/3 +i = q -
\left( (q+1)/3 - i-1 \right)$.  From the above comments, the element
$a$ will occur $\lceil \frac{i}{2} \rceil$ times in sets in $T$ and
the element $b$ will occur $\lfloor \frac{(q+1)/3- i-1}{2}\rfloor$ times.
So there are $\lceil \frac{i}{2} \rceil + \lfloor \frac{(q+1)/3-
i-1}{2}\rfloor = (q-2)/6$ elements $a$ in sets in $T$ with $a \equiv
3i \pmod{q+1}$.

If $q$ is even, each element from $\{1,\dots, (q+1)/3\}$ occurs in
$(q-2)/3$ triples, and for every $i$ there are in total, $(q-2)/6$ elements $a$
in triples from $T$ with $a \equiv 3i \pmod{q+1}$. In this case we have that
\begin{align*}
\sum_{(k,\ell,m) \in T} e^{3uk}+e^{3u\ell}+e^{3um}
= \frac{q-2}{3} \sum_{i=1}^{\frac{q+1}{3}}e^{3ui} + \frac{q-2}{6} \sum_{i=1}^{\frac{q+1}{3}}e^{3ui}
= 0.
\end{align*}

Next assume that $q$ is odd. If $i$ is even, then in total there are
$(q-5)/6$ elements $a$ in triples from $T$ for which $a \equiv
3i \pmod{q+1}$.  If $i$ is odd, then in total there are
$(q+1)/6$ elements $a$ in triples from $T$ for which $a \equiv
3i \pmod{q+1}$. Thus for $q$ even
\begin{align*}
\sum_{(k,\ell,m) \in T} e^{3uk}+e^{3u\ell}+e^{3um}
= \frac{q-2}{3} \sum_{i=1}^{\frac{q+1}{3}}e^{3ui} + \frac{q-5}{6} \sum_{i=1}^{\frac{q+1}{6}}e^{6iu} + \frac{q+1}{6} \sum_{i=1}^{\frac{q+1}{6}}e^{3u(2i-1)}
= 0.
\end{align*}
\qed

With this lemma, it is easy to determine the eigenvalue for the
characters $\chi_3$ and $\chi_4$.

\begin{lem}
The eigenvalue of $\Gamma_2$ for any of the characters of type
$\chi_3$ and $\chi_4$ is $0$.
\end{lem}

\begin{lem}
Let $\omega$ be a third root of unity. Then 
\[
\sum_{(k,\ell,m) \in T}(\omega^{k-\ell} + \omega^{\ell-k}) = 
\begin{cases}
 -\frac{2q+2}{3},   & {\textrm{if }  q+1 \equiv 0 \pmod 9;}\\
 -\frac{2q-4}{3},    & {\textrm{if }  q+1 \equiv 3 \pmod 9;}\\ 
 -\frac{2q-13}{3}, & {\textrm{if }  q+1 \equiv 6 \pmod 9.}
\end{cases}
\]
\end{lem}
\proof 
For any $k$ and $\ell$, the set $\{ (k-\ell) \pmod{3}, (\ell-k)\pmod{3}\}$ will either be
$\{1,2\}$ or $\{0,0\}$.  If the value of these differences is equal to
$(0,0)$, then $\omega^{k-\ell}+\omega^{\ell-k}$ equals $2$, otherwise it is
equal to $-1$.

Count the number of times that this set is $\{0,0\}$. Any such set has
the form $(k,\ell) = (a, a+3i)$ with $a \in \{1, \dots, \frac{q+1}{3}
- 3i\}$ for some $i \leq \lfloor \frac{q+1}{9}\rfloor$. Thus there are
\[
N_{(0,0)}:=\sum_{i=1}^{\lfloor \frac{q+1}{9}\rfloor} \left( \frac{q+1}{3} - 3i \right)
= \left\lfloor \frac{q+1}{9} \right\rfloor  \frac{q+1}{3} 
- \frac{3}{2} \left(  \left\lfloor \frac{q+1}{9}\right\rfloor \left( \left \lfloor \frac{q+1}{9}\right \rfloor+1 \right) \right) 
\]
pairs $(k,\ell)$ in which the difference is equivalent to $\{0,0\}$. The
number of pairs $(k,\ell)$ in which the difference is equivalent to
$\{1,2\}$ is then
\[
 \frac{(q+1)(q-2)}{18} - N_{(0,0)}.
\]

Thus the sum of $(\omega^{k-\ell} + \omega^{\ell-k})$ over all $k$ and
$\ell$ with $(k,\ell,m) \in T$ is equal to
\begin{align*}
& -1 \left( \frac{(q+1)(q-2)}{18} - N_{(0,0)} \right)  + 2N_{(0,0)}  \\
&=\left( - \frac{(q+1)(q-2)}{18} 
   +3 \left( \left\lfloor \frac{q+1}{9}\right\rfloor  \frac{q+1}{3} - \frac{3}{2} 
              \left( \left \lfloor \frac{q+1}{9}\right\rfloor
                \left(\left\lfloor \frac{q+1}{9}\right\rfloor-1 \right) \right)  \right) \right) \\
&= \begin{cases}
 -\frac{2q+2}{3},   & {\textrm{if }  q+1 \equiv 0 \pmod 9;}\\ 
 -\frac{2q-4}{3},    & {\textrm{if }  q+1 \equiv 3 \pmod 9;}\\
 -\frac{2q-13}{3}, & {\textrm{if }  q+1 \equiv 6 \pmod 9.}
\end{cases}
\end{align*}
\qed

With the previous result, it is straight-forward to calculate the
eigenvalue for $\Gamma_2$ for the irreducible characters of type
$\chi_5$.

\begin{lem} 
The eigenvalue of $\Gamma_2$ for $\chi_5$ is
\[
\lambda_{\chi_5} =
\begin{cases}  \vspace{.2cm}
\frac{6 |G| }{ (q-1)(q+1)(q^2-q+1)}, & {\textrm{if } q+1 \equiv 0
  \pmod{9}; } \\  \vspace{.2cm}
\frac{ 6(q-2) |G| } {(q-1)(q^2-q+1) (q+1)^2}, & {\textrm{if }q+1\equiv
  3\pmod{9};}  \\ 
\frac{3|G|(2q-13)}{ (q-1)(q+1)^2(q^2-q+1)},  & {\textrm{if } q+1 \equiv 6 \pmod{9}. } \\
\end{cases}
\]
\end{lem}

Next we find the eigenvalues for the characters of type $\chi_6$.
Each of these characters is parameterized by a triple $(u,v,w) \in T$
and each conjugacy class of type $C_2$ is parameterized by a triple
$(k,\ell,m) \in T$. Again we let $S$ denote the set of all
permutations of all triples from $T$. Next we prove two lemmas, from
these it is straight-forward to calculate the eigenvalues of
$\Gamma_2$ that correspond to the characters of type $\chi_6$.

\begin{lem}\label{lem:firstchi6}
Let $(u,v,w) \in T$. If $v-u \equiv 0 \pmod {3}$, then the sum
\[
\sum_{(k,\ell,m) \in S} e^{uk+ v\ell+ wm}
\]
is equal to $-(q+1)/3$, provided that one of $u, v$ or $w$ is
divisible by $(q+1)/3$; otherwise it is equal to $0$.
\end{lem}
\proof Assume that $v-u \equiv 0 \pmod{3}$. Since $(u,v,w) \in
T$ we know that $3$ divides $u+v+w$, which implies that both $u-w$
and $v-w$ are divisible by $3$.

We will count the number of solutions $(k,\ell,m) \in S$ to 
\begin{align}\label{eq:dot}
uk+ v\ell+ wm \equiv i \pmod{q+1}
\end{align}
for each $i \in \{0,\dots,q\}$.  Since $k+\ell+m \equiv 0 \pmod{q+1}$,
Equation~\ref{eq:dot} can be reduced to
\[
(u-w)k+ (v-w)\ell \equiv i \pmod{q+1}.
\]
Both $u-w$ and $v-w$ are divisible
by $3$, so Equation~\ref{eq:dot} has solutions if and only if $i$ is a multiple
of $3$. In this case, the equation can be further reduced to
\[
\frac{(u-w)k}{3}+ \frac{(v-w)\ell}{3} \equiv i/3 \pmod{\frac{q+1}{3}}.
\]
Set 
\[
d =\ gcd\left( \frac{v-w}{3}, \, \frac{v-w}{3}, \, \frac{q+1}{3} \right) = \frac{1}{3} \gcd(u,v,w,q+1).
\]
Using Proposition~\ref{prop:numbersoln}, Equation~\ref{eq:dot} has $d(q+1)/3$
solutions $k,\ell \in \{1,\dots, (q+1)/3\}$.  Further, set 
\[
d_u = \gcd(u, \frac{q+1}{3}),\, d_v = \gcd(v, \frac{q+1}{3}),\, d_w = \gcd(w,\frac{q+1}{3}).
\]
As in the proof of Lemma~\ref{lem:chi5}, if $d_w$ divides $i/3$, then
$d_w$ of these solutions have $k=\ell$.

Counting the solutions with $k,m\leq (q+1)/3 < \ell$ and $\ell, m \leq
(q+1)/3 < k$, the total number of solutions to Equation~\ref{eq:dot} is
\[
3d \left( \frac{q+1}{3} \right) - \sum_{\stackrel{j \in\{u,v,w\}}{d_j|\frac{i}{3}}} d_j 
\]

As in the proof of Lemma~\ref{lem:chi5}, this implies that
\begin{align*}
 \sum_{(k,l,m) \in S} e^{uk+v\ell+wm} 
&= - d_u \sum_{\stackrel{i\in\{1,\dots, q+1\}}{d_u | \frac{i}{3}}}e^{i} 
   - d_v \sum_{\stackrel{i\in\{1,\dots, q+1\}}{d_v | \frac{i}{3}}}e^{i} 
   - d_w \sum_{\stackrel{i\in\{1,\dots, q+1\}}{d_w | \frac{i}{3}}}e^{i}.
\end{align*}
If the triple $\{u,v,w\}$ includes a multiple of $(q+1)/3$, then this is equal to
$-(q+1)/3$, otherwise it is equal to $0$.  
\qed

\begin{lem}\label{lem:secondchi6}
Let $(u,v,w) \in T$. If $v -u \not \equiv 0 \pmod {3}$, then the sum
\[
\sum_{(k,\ell,m) \in S} e^{uk+ v\ell+ um}
\]
is equal to $(q-8)/3$ provided that one of $u,v$ or $w$ is divisible
by $(q+1)/3$; otherwise it is equal to $3$.
\end{lem}
\proof 
Once again, we will count the number of solutions $(k,\ell,m)$ to
\begin{align}\label{lasteq} 
uk+v\ell+wm \equiv i \pmod{q+1}
\end{align} 
with exactly two of $k,\ell,m$ in
$\{1,\dots,(q+1)/3 \}$ and $k+\ell+m \equiv q+1 \pmod{q+1}$.

Set $d = \gcd(u,v, (q+1)/3) = \gcd(u,v,w,(q+1)/3)$.
Similar to the proof of Lemma~\ref{lem:chi5}, there are $d(q+1)$
solutions to Equation~\ref{lasteq}.

If $(k,\ell,m)$ is a solution to Equation~\ref{lasteq}, then both
\[
(k+\frac{q+1}{3},\, \ell+\frac{q+1}{3},\, m+\frac{q+1}{3}), 
\quad (k+\frac{2(q+1)}{3},\, \ell+\frac{2(q+1)}{3},\, m+\frac{2(q+1)}{3})
\]
(taken modulo $q+1$) are also solutions to Equation~\ref{lasteq}. We
will call these the three \textsl{shifted solutions}.  It is clear that exactly
one of $k$, $k+(q+1)/3$, and $k+2(q+1)/3$ is in the range $\{1, \dots,
(q+1)/3\}$ (and obviously the same holds for $\ell$ and $m$).

If $i=0$, then $k=\ell=m = (q+1)/3$ is a solution, if fact it is the
only solution in which all three of $k,\ell$ and $m$ are no larger
than $(q+1)/3$. Further, $k \equiv \ell \equiv m \equiv 0 \pmod{
  (q+1)/3}$ is a solution when $i=0$. We claim that these solutions
are the only ones in which each of the three shifted solutions has
exactly one element less than $(q+1)/3$.  To see this, assume without
loss of generality that $k \in \{1,\dots, (q+1)/3\}$, $\ell \in
\{(q+1)/3+1, \dots, 2(q+1)/3\}$ and $m \in \{2(q+1)/3+1, \dots,
q+1\}$. But this implies that
\[
1+(q+1)/3+1 + 2(q+1)/3 +1 = q+4\leq k+\ell+m ,
\]
so $k+\ell+m = 2(q+1)$. This happens if and only if $k = (q+1)/3$,
$\ell=2(q+1)/3$ and $m=q+1$.

So, other than these solutions with $i=0$, exactly one of the three
shifted solutions will have only one of $k,\ell,m$ larger than
$(q+1)/3$.  Thus, of the $d(q+1)$ solutions to Equation~\ref{lasteq}
(other than then the examples above with $i=0$), there are $d(q+1)/3$
solutions $(k,\ell,m)$ to with exactly one of $k,\ell$ or $m$ larger
than $(q+1)/3$. Next we count the number of these solutions with
$k=\ell$, $k=m$ or $\ell=m$ to determine all the solutions that are in
$S$. Similar to Lemma~\ref{lem:chi5}, we set
\[
d_u =\gcd(u, \frac{q+1}{3}),\quad d_{v} =\gcd(v, \frac{q+1}{3}), \quad d_{w}=\gcd(w, \frac{q+1}{3}).
\]
Assume that $k=\ell$, then 
\begin{align*}
uk+v\ell+wm &\equiv k(u+v) +wm \pmod{q+1} \\
            &\equiv k(q+1-w) + (q+1-2k)w \pmod{q+1} \\
            &\equiv k(-w) + -2 kw \pmod{q+1} \\
            &\equiv -3 um \pmod{q+1} 
\end{align*}
Thus Equation~\ref{lasteq} has a solution with $k=\ell$ if and only if
$3|i$. In this case the equation becomes
\[
-um \equiv i/3 \pmod{ \frac{q+1}{3}}.
\]
If $d_w$ divides $i/3$, then there are $d_w$ solutions to
Equation~\ref{lasteq} with $k=\ell$.  Similarly, if $d_v$ divides
$i/3$, there are $d_v$ solutions with $k=m$; and if $d_u$ divides
$i/3$, there are $d_u$ solutions with $\ell=m$.

If $i \not \equiv 0 \pmod{(q+1)/3}$ then the number of solutions $(k,\ell,m)$ to
Equation~\ref{eq:dio} with $(k,\ell,m) \in S$ is
\[
\frac{d(q+1)}{3} - \sum_{\stackrel{u,v,w}{d_j | i/3}} d_j.
\]
If $i = (q+1)/3$ or $2(q+1)/3$ then the number of solutions is
\[
\frac{ d(q+1)}{3} - 1 -  \sum_{\stackrel{u,v,w}{d_j | i/3}} d_j
\]
since we do not count the solutions with $ k \equiv \ell \equiv m
\equiv 0 \pmod{q+1}$.

If $i =0$ then the number of solutions is
\[
\frac{d(q+1)}{3} - d_u -d_v -d_w +2,
\]
(since we have removed the solution with $k=\ell=m=(q+1)/3$ three
times, rather than once we need to add $2$).

For an integer $i$ define $f(i)$ to be the set of $j \in \{u,v,w\}$
for which $d_j$ divides $i/3$.  Thus we have that
\begin{align*}
 \sum_{[k,\ell,m] \in S} e^{uk+v\ell+wm} 
 &= d \frac{q+1}{3} \sum_{i}  e^{i} - 
\left(    (d_u+d_v+d_w)\sum_{\stackrel{i}{f(i) = \{d_u,d_v,d_w\}}}\!\!
  e^{i}  \right. \\
&+    (d_u+d_v)\!\!\sum_{\stackrel{i}{f(i) = \{d_u,d_v\}}} \!\!e^{i} 
+    (d_u+d_w)\!\!\sum_{\stackrel{i}{f(i) = \{d_u,d_w\}}}\!\! e^{i} 
+    (d_v+d_w)\!\!\sum_{\stackrel{i}{f(i) = \{d_v,d_w\}}}\!\! e^{i}   \\
&\left.+    d_u\!\!\sum_{\stackrel{i}{f(i) = \{d_u\}}} \!\!e^{i} 
+    d_v\!\!\sum_{\stackrel{i}{f(i) = \{d_v\}}} \!\!e^{i} 
+    d_w\!\!\sum_{\stackrel{i}{f(i) = \{d_w\}}} \!\! e^{i} + \sum_{i=0}^2 e^{\frac{i(q+1)}{3}} \right)  \\
&+ 2e^0 - e^{(q+1)/3} - e^{2(q+1)/3} \\
&=3 - d_u \sum_{d_u | i/3}e^{i} -d_v \sum_{d_v|i/3}e^{i} - d_w \sum_{d_w|i/3}e^{i}.
\end{align*}

If the triple $\{u,v,w\}$ includes a multiple of $(q+1)/3$, then this is equal to
$3-(q+1)/3$, otherwise it is equal to $3$.  \qed

\begin{lem}
Let $(u,v,w) \in T$. 
If $v-u \equiv 0 \pmod {3}$, then  the $\chi_6$ eigenvalue of
$\Gamma_2$ is either $0$, or
\[
-\frac{|G|}{(q+1)(q-1)(q^2-q+1)}
\] 

If $v-u \not \equiv 0 \pmod {3}$, then the
  $\chi_6$ eigenvalue of $\Gamma_2$ is either 
\[
-\frac{|G|(q-8)}{(q+1)^2(q-1)(q^2-q+1)}
\]
or
\[
\frac{9|G|}{(q+1)^2(q-1)(q^2-q+1)}.\qed
\]
\end{lem}

Next we consider the character $\chi_6$, this character is indexed by
the elements in $T$.

\begin{lem}
For every triple $(u,v,w)\in T$, the $\chi_6$ eigenvalue of
$\Gamma_3$ is either
\[
-\frac{6|G|}{(q-1)(q+1)^2(q^2-q+1)} 
\]
or
\[
\frac{3|G|}{(q-1) (q+1)^2 (q^2-q+1)}.
\]
\end{lem}
\proof If $u-v \equiv 0 \pmod{3}$ then
$-3(\omega^{u-v}+\omega^{v-u}) = -6$, the $\chi_6$ eigenvalue
for $\Gamma_3$ is
\[
-\frac{6|G|}{(q+1)^2(q-1)(q^2-q+1)} 
\]

On the other hand, when $u-v \not\equiv 0 \pmod{3}$ it follows that
$-3(\omega^{u-v}+\omega^{v-u}) = 3$. In this case the $\chi_6$ eigenvalue of
$\Gamma_3$ is
\[
\frac{3|G|}{(q-1)(q^2-q+1) (q+1)^2}.
\]
\qed

The final character to consider is $\chi_8$.

\begin{lem}
The eigenvalue of $\Gamma_1$ for  $\chi_8$ eigenvalue is 
\[
\frac{3|G|}{(q^2-q+1) (q+1)^2(q-1)}.
\]
\end{lem}
\proof By orthogonality of the characters, the sum of the values of
$B$ over all the conjugacy classes of type $C_1$ is equal to $1$.
Thus the eigenvalue is
\[
\frac{1}{(q+1)^2(q-1)} \frac{3|G|}{(q^2-q+1)}
=\frac{3|G|}{(q^2-q+1) (q+1)^2(q-1)}
\]
\qed

We summarize the results from this section in Table~\ref{tab:evalues3divides}. The
character $\chi_5$ when $q+1 \equiv i \pmod {9}$ is denoted by $\chi_5^i$.

\begin{table}[t]
\begin{tabular}{| l | c | c | c| c |} \hline
Character& Number & eigenvalue & eigenvalue & eigenvalue\\ 
  &   &  of $\Gamma_1$ & of $\Gamma_2$ & of $\Gamma_3$ \\ \hline
$\one$& 1 &  $\frac{|G|(q^2-q-2)}{3(q^2-q+1)}$ & $\frac{|G|(q-2)}{6(q+1)}$ & $\frac{|G|}{(q+1)^2}$  \\ \hline

$\chi_1$ & 1 & $-\frac{|G|(q^2-q-2)}{3q(q-1)(q^2-q+1)}$ &
$\frac{|G|(q-2)}{3q(q-1)(q+1)}$ & $\frac{2|G|}{q(q-1)(q+1)^2}$
\\ \hline

$\chi_2$ & 1 & $-\frac{|G|(q^2-q-2)}{3q^3(q^2-q+1)}$ &
$-\frac{|G|(q^2-q-2)}{6q^3(q+1)^2}$ &  $\frac{-|G|}{q^3(q+1)^2}$ \\ \hline 

$\chi_3$ &$\frac{q-2}{3}$& $0$ & $0$ &
$\frac{3|G|}{(q^2-q+1)(q+1)^2}$ \\ \hline

$\chi_4$ &$\frac{q-2}{3}$& $0$ & $0$ &
$-\frac{3|G|}{q(q^2-q+1)(q+1)^2}$ \\ \hline

$\chi_5^{\tiny{0}}$ & 3  &0 &
$\frac{6|G|}{(q+1)(q-1)(q^2-q+1)} $ 
  &  $-\frac{6|G|}{(q-1)(q+1)^2(q^2-q+1)}$ \\ \hline  

$\chi_5^{\tiny{3}}$ & 3  &0 &
$\frac{6|G|(q-2)}{(q+1)^2(q-1)(q^2-q+1)}$ &
 $\frac{3|G|}{(q-1)(q+1)^2(q^2-q+1)}$\\ \hline

$\chi_5^{\tiny{6}}$ & 3  &0 &
$\frac{3|G|(2q-13)}{ (q-1)(q+1)^2(q^2-q+1)}$ 
  &  $\frac{3|G|}{(q-1)(q+1)^2(q^2-q+1)}$ \\ \hline 

 $\chi_6$ &  & 0&0& $-\frac{6|G|}{(q-1)(q^2-q+1) (q+1)^2}$\\ \hline
$\chi_6$ &  & 0&
 $-\frac{|G|}{(q+1)(q-1)(q^2-q+1)}$ & $-\frac{6|G|}{(q+1)^2(q-1)(q^2-q+1)}$\\ \hline
 $\chi_6$ &  & 0& $-\frac{|G|(q-8)}{(q-1)(q+1)^2(q^2-q+1)}$ & $\frac{3|G|}{(q-1)(q+1)^2(q^2-q+1)}$\\ \hline
 $\chi_6$ &  & 0&$\frac{9|G|}{(q-1)(q+1)^2(q^2-q+1)}$ &$ \frac{3|G|}{(q-1)(q+1)^2(q^2-q+1)}$\\\hline

 $\chi_7$ & $\frac{q^2-q-2}{6}$ & 0 &0& 0 \\ \hline
$\chi_8$ & $\frac{q^2-q-2}{18}$  & $\frac{3|G|}{(q^2-q+1) (q+1)^2(q-1)}$ & 0& 0\\ \hline
\end{tabular}
\caption{The eigenvalues for the derangement graphs for $\psu(3,q)$
  where $3 | q+1$~\label{tab:gcd3}. The number of the eigenvalues for
  each character of type $\chi_6$ is described in
  Lemma~\ref{lem:firstchi6} and~\ref{lem:firstchi6}.\label{tab:evalues3divides}}
\end{table}

\section{Eigenvalues of the Adjacency Matrices}
\label{sec:evaluesAdj}

In this section we will show that Lemma~\ref{lem:lotsareperm} applies
to the group $\psu(3,q)$. The first step is to calculate all the
eigenvalues for $\Gamma_{\psu(3,q)}$. Using the the previous tables,
this is straight-forward for all values of $q$ and we record the
results in Tables~\ref{tab:evalueDer} and~\ref{tab:der3divides}.


\begin{table}
\begin{tabular}{| l | c | c | c |} \hline
Character& Number & eigenvalue \\ 
  &   &  of $\Gamma_{\psu(3,q)}$ \\ \hline
 Trivial& 1 & $\frac{|G|(q^2-q) (q^2+q+1)}{2(q^2-q+1)(q+1)^2}$\\
 $\chi_1$ & 1 & $ -\frac{|G|q} {(q^2-q+1)(q+1)^2}$ \\
$\chi_2$ & 1 &  $-\frac{|G|(q-1)(q^2+q+1)}{2q^2(q+1)^2(q^2-q+1)}$ \\
$\chi_3$, $u = \frac{q+1}{2}$ & $q+1$ & $-\frac{|G|(q+1)}{2(q+1)^2(q^2-q+1)}$\\
$\chi_3$, $u \neq  \frac{q+1}{2}$ & $1$ & $ \frac{|G|}{(q^2-q+1)(q+1)^2} $\\
$\chi_4$, $u = \frac{q+1}{2}$ & $q+1$ &$\frac {|G|  (q-1)}{ 2q (q^2-q+1)(q+1)^2}$\\
$\chi_4$, $u \neq \frac{q+1}{2}$ & $1$ &$-\frac{|G|}{q(q^2-q+1)(q+1)^2}$\\
$\chi_5$ & $\frac{q^2-4q}{6}$ & $-\frac{|G|}{(q+1)^2(q^2-q+1)}$ \\
$\chi_5$ & $\frac{q}{2}$ &  $ -\frac{2|G|}{(q+1)^2(q-1) (q^2-q+1)}$\\
$\chi_6$ & $\frac{q^2-q-2}{2}$ &  $0$ \\
$\chi_7$ & $\frac{q^2-q}{6}$ & $\frac{|G|}{(q^2-q+1)(q+1)^2(q-1)}$ \\  \hline
\end{tabular}
\caption{Eigenvalues for the derangement graph $\psu(3,q)$ where $3 \notdivide q+1$.~\label{tab:evalueDer}}
\end{table}

\begin{table}
\begin{tabular}{| l |  c | c | c |} \hline
Character& Number & Dimension & eigenvalue \\ 
  &  &   &  of $\Gamma$ \\ \hline
$\one$& 1 & 1 & $\frac{|G|(q^4-5q)}{2(q^2-q+1)(q+1)^2}$   \\ \hline
$\chi_1$ & 1 & $q(q-1)$ & $-\frac{|G|(q^3-3q^2-2)}{q(q-1)(q+1)^2(q^2-q+1)}$ \\ \hline
$\chi_2$ & 1 & $q^3$ & $-\frac{|G|(q^3-5)}{2q^2(q^2-q+1)(q+1)^2}$ \\ \hline
$\chi_3$ &$\frac{q-2}{3}$& $q^2-q+1$& $\frac{3|G|}{(q^2-q+1)(q+1)^2}$ \\  \hline
$\chi_4$ &$\frac{q-2}{3}$& $q(q^2-q+1)$&  $-\frac{3|G|}{q(q^2-q+1)(q+1)^2}$ \\ \hline


$\chi_5^{\tiny{0}}$ & 3 & $\frac{(q-1)(q^2-q+1)}{3}$ & $\frac{6q|G|}{(q+1)^2
  (q-1) (q^2-q+1)}$ \\ \hline 
$\chi_5^{\tiny{3}}$ & 3 & $\frac{(q-1)(q^2-q+1)}{3}$ & $\frac{(6q-9)|G|}{(q+1)^2
  (q-1) (q^2-q+1)}$ \\ \hline 
$\chi_5^{\tiny{6}}$ & 3 & $\frac{(q-1)(q^2-q+1)}{3}$
& $\frac{(6q-42)|G|}{ (q-1)(q+1)^2(q^2-q+1)}$ \\ \hline  

$\chi_6$ &  & $(q-1)(q^2-q+1)$ & $-\frac{6|G|}{(q-1)(q+1)^2(q^2-q+1)}$\\ \hline
$\chi_6$ &  &$(q-1)(q^2-q+1)$ &$-\frac{|G|(q+7)}{(q-1)(q+1)^2(q^2-q+1)}$\\ \hline
$\chi_6$ &  &$(q-1)(q^2-q+1)$ &$-\frac{|G|(q-11)}{(q-1)(q+1)^2(q^2-q+1)}$\\ \hline
$\chi_6$ &  & $(q-1)(q^2-q+1)$ & $\frac{12|G|}{(q-1)(q+1)^2(q^2-q+1)}$\\ \hline

$\chi_7$ & $\frac{q^2-q-2}{6}$ & $(q+1)(q^2-q+1)$ & 0 \\ \hline
$\chi_8$ & $\frac{q^2-q-2}{18}$  & $(q+1)^2(q-1)$ &  $\frac{3|G|}{(q^2-q+1) (q+1)^2(q-1)}$  \\ \hline
\end{tabular}
\caption{Eigenvalues for the derangement graph $\psu(3,q)$ where $3  | q+1$.~\label{tab:der3divides}}
\end{table}

When $3 \notdivide q$, the least eigenvalue of $\Gamma_{\psu(3,q)}$ is
$-\frac{q|G|}{(q+1)^2 (q^2-q+1)}$ and is given by the
character of type $\chi_1$. Applying Hoffman's ratio bound gives
\[
\alpha(\Gamma_q) \leq \frac{|G|}{1 -\frac{\frac{|G|(q^2-q)(q^2+q+1)}  {2(q+1)^2 (q^2-q+1)} }{-\frac{q|G|}{(q+1)^2 (q^2-q+1)}}} = 
  \frac{2|G|}{2 +(q-1)(q^2+q+1)} 
= \frac{2|G|}{(q^3+1)}
= 2q^3(q+1)(q-1)
\]
But this is not the size of the largest set of intersecting
permutations. 

The eigenvalue for $\chi_2$, while not the smallest, has the property
that it is equal to the eigenvalue for the trivial character (so the
degree) divided by $-q^3$. Further, the character $\chi_2$ is
the only character that gives this eigenvalue.

When $3 | q+1$, the least eigenvalue of $\Gamma_{\psu{3,q}}$ belongs
to $\chi_1$ and is equal to
$-\frac{|G|(q^3-3q^2-2)}{q(q-1)(q+1)^2(q^2-q+1)}$ and, again,
Hoffman's bound does not hold with equality for the adjacency matrix.
Again, the eigenvalue for $\chi_2$ is equal to the eigenvalue for the
trivial character (so the degree) divided by $-q^3$, and provided that
$q\neq 5$, this is the only character for which this is true.

The next step is to give a weighted adjacency matrix for
$\Gamma_{\psu(3,q)}$ for which Hoffman's ratio bound does hold with
equality. This weighting was given in~\cite{MR3474795}.

In the case where $3 \notdivide q+1$, we will assign a weight of $a$
on the conjugacy classes in the family $C_1$, and a weight of $b$ on
the conjugacy classes of type $C_2$. Let $A$ be the weighted adjacency
matrix for the derangement graph for $\psu(3,q)$ with this
weighting. The $(\sigma,\pi)$-entry of $A$ is $a$ if $\sigma^{-1}\pi$
is in one of the conjugacy classes of $C_1$, the entry is $b$ if
$\sigma^{-1}\pi$ is in one of the conjugacy classes of $C_2$ and any
other entry is equal to $0$.

The eigenvalues for the characters $\chi_1$ and $\chi_2$
(respectively) of $A$ (when $3  \notdivide q+1$) are
\[
\frac{1}{q(q-1)}  \left(  (-1)a\frac{q^2-q}{3}\,\frac{|G|}{q^2-q+1} 
+  2b  \frac{(q^2-q)}{6}\frac{|G|}{(q+1)^2}  \right) 
\]
and
\[
\frac{1}{q^3} \left( (-1) a \frac{ q^2-q}{3}\frac{|G|}{q^2-q+1} 
+ (-1) b \frac{q^2-q}{6 } \frac{|G|}{(q+1)^2}  \right).
\]
Solving for $a$ and $b$ so that these eigenvalues are both $-1$ gives 
\[
a= \frac{(q^2-q+1)(2q^2+q-1)}{|G| (q-1)}, 
\quad 
b=\frac{ 2 (q^2-q+1)(q+1)^2}{|G| (q-1)}.
\]
The eigenvalues of the weighted adjacency matrix $A$ can be calculated directly from
Table~\ref{tab:gcd1}.

\begin{table}
\begin{tabular}{|l| c | c | c |} \hline
Type & Number & Dimension & eigenvalue \\ \hline
Trivial & 1 & 1 & $q^3$ \\
$\chi_1$ & 1 & $q^3$ & -1 \\ 
$\chi_2$ & 1 & $q(q-1)$ & -1 \\
$\chi_3$, $u = \frac{q+1}{2}$ & 1 & $q^2-q+1$ & -1 \\
$\chi_3$, $u \neq \frac{q+1}{2}$ &$q-1$ & $q^2-q+1$ & $\frac{2}{q-1}$ \\
$\chi_4$, $u = \frac{q+1}{2}$ & 1 & $q(q^2-q+1)$ & $\frac{1}{q}$ \\
$\chi_4$, $u \neq \frac{q+1}{2}$ &$q-1$ & $q(q^2-q+1)$ & $-\frac{2}{q(q-1)} $ \\   
$\chi_5$ & $\frac{(q-1)(q-3)}{6}$ & $(q-1)(q^2-q+1)$ & $\frac{-4}{(q-1)^2} $\\
$\chi_5$ &$(q-1)/2$& $(q-1)(q^2-q+1)$ &  $-\frac{4}{(q-1)^2}$ \\
$\chi_6$ & $\frac{q^2-q-2}{2}$ & $(q+1)(q^2-q+1)$ & $0$ \\
$\chi_7$ &$\frac{q^2-q}{6}$ & $(q+1)^2(q-1)$ & $\frac{2q-1}{(q-1)^2(q+1)}$  \\ \hline
\end{tabular}
\caption{Table of eigenvalues of the weighted adjacency matrix for $3
  \notdivide q+1$.}
\end{table}

Similarly, in the case where $3| q+1$, we find a weighting so that the
eigenvalues corresponding to characters $\chi_1$ and $\chi_2$ are both
equal to $-1$.  We use the weight $a$ for the conjugacy classes of
type $C_1$ and another weight, $b$, for any conjugacy class of type
$C_2$ or $C_3$. Set the value of $a$ and $b$ to be
\[
a= \frac{q(q^2-q+1)(2q^2+q-1)}{|G| (q^2-q-2)},
\quad 
b= \frac{2q(q+1)^2(q^2-q+1)}{|G| (q^2-q+4)}.
\]
The eigenvalues of this weighted adjacency matrix can be calculated
directly from the values in Table~\ref{tab:gcd3}.

\begin{table}
\begin{tabular}{| l | c | c | c |} \hline Character &Number & Dimension
  & eigenvalue \\ \hline
  trivial & 1 & 1 & $q^3$ \\
  $\chi_1$ & 1 & $q(q-1)$ & $-1$ \\
  $\chi_2$ & 1 & $q^3$ & $-1$ \\
  $\chi_3$ & $\frac{q-2}{3}$ & $(q^2-q+1)^2$ & $\frac{6q}{q^2-q+4}$ \\
  $\chi_4$ & $\frac{q-2}{3}$ & $q^2(q^2-q+1)^2$ & $\frac{-6}{q^2-q+4}$ \\
  $\chi_5^{\tiny{0}}$ & 3 & $(q-1)^2(q^2-q+1)^2/9$ & $\frac{2q(6q)}{(q-1)(q^2-q+4)}$\\ 
  $\chi_5^{\tiny{3}}$ & 3 & $(q-1)^2(q^2-q+1)^2/9$ & $\frac{2q(6q-9)}{(q-1)(q^2-q+4)}$\\ 
  $\chi_5^{\tiny{6}}$ & 3 & $(q-1)^2(q^2-q+1)^2/9$ & $\frac{2q(6q-42)}{(q-1)(q^2-q+4)}$\\ 

  $\chi_6$& &$(q-1)^2(q^2-q+1)^2$ &  $-\frac{12q}{(q-1)(q^2-q+4)}$ \\
  $\chi_6$& &$(q-1)^2(q^2-q+1)^2$ &  $-\frac{ 2q(q+6)(-6)}{(q-1)(q^2-q+4)}$ \\
  $\chi_6$& &$(q-1)^2(q^2-q+1)^2$ &  $-\frac{2q(q-11)(-6)}{(q-1)(q^2-q+4)}$ \\
  $\chi_6$& &$(q-1)^2(q^2-q+1)^2$ &  $\frac{24q(-6)}{(q-1)(q^2-q+4)}$ \\

  $\chi_7$ & $\frac{q^2-q-2}{6}$ & $(q+1)^2(q^2-q+1)^2$ & $0$ \\
  $\chi_8$ & $\frac{q^2-q-2}{18}$ & $(q+1)^4(q-1)^2$& $\frac{3q(2q-1)}{(q-1)(q+1)(q^2-q-2)}$  \\ \hline
\end{tabular}
\caption{Table of eigenvalues of the weighted adjacency matrix for $3 | q+1$.}
\end{table}

Applying Hoffman's ratio bound to the weighted adjacency matrix, we get
in both cases that
\begin{eqnarray}\label{eq:tightbound}
\alpha(\Gamma_{\psu(3,q)}) = \frac{|\psu(3,q)|}{q^3+1}.
\end{eqnarray}
Equality holds since any canonical set meets this bound. This
means that the group $\psu(3,q)$ has the EKR property (this is shown
in~\cite{MR3474795}).  By Tables~\ref{tab:evalueDer} and
~\ref{tab:der3divides}, provided that $q\neq 5$, only $\chi_2$ gives
the eigenvalue that gives equality in the equation for Hoffman's
bound.  Thus $\psu(3,q)$ satisfies all the conditions of
Lemma~\ref{lem:lotsareperm}, we  conclude that every maximum coclique
is in the module.

\begin{thm}
For all $q\neq 5$, the group $\psu(3,q)$ has the EKR-module property.
\end{thm}

\section{Further work}

We have shown that the group $\psu(3,q)$ has the EKR-module
property. This means that the characteristic vector for any maximum
intersecting set of permutations from $\psu(3,q)$ is a linear
combination of the characteristic vectors of the canonical
intersecting sets. It is still open if the group $\psu(3,q)$ has the
strict-EKR property, but the result in this paper gives an interesting
property of the intersecting sets from $\psu(3,q)$ which may be be
useful for characterizing the maximum intersecting sets.

For example, the method described in~\cite{AhMe} proves that a group
with the EKR-module property has the strict-EKR property, provided
that a certain matrix has full rank. For $q=2$ this matrix does not
have full rank. The derangement graph for $\psu(3,2)$ is 8 disjoint
copies of $K_9$, so it is clear what the maximum cocliques are and
that $\psu(3,2)$ does not have the strict-EKR property. For $q=3$, a
calculation (using GAP~\cite{GAP}) shows that this matrix does have
full rank; thus $\psu(3,3)$ does have the strict-EKR property. For
larger values of $q$, a more efficient method to find the rank of this
smaller matrix would be needed to prove that $\psu(3,q)$ has the
strict-EKR property. 

Although all 2-transitive groups have the ERK-property, it is not true
that all 2-transitive groups have the strict-EKR property. It has been
shown that $\sym(n)$, $\alt(n)$, $\pgl(2,q)$, $\pgl(3,q)$,
$\psl(2,q)$~\cite{long} and the Mathieu groups all have the strict-EKR
property by first showing that they have the EKR-module property
~\cite{AhMeAlt, AhMe, GoMe, KaPa1, KaPa2, long}. It is shown
in~\cite{MR3474795} that the Suzuki groups, Ree Groups, Higman-Sims,
Symplectic groups also have the EKR-module property. The result in
this paper shows that every minimal almost-simple 2-transitive groups
(except possibly $\psu(3,5)$) has the EKR-module property. Thus we
conclude with the following conjecture.

\begin{conj}
Every $2$-transitive group has the EKR-module property.
\end{conj}

\begin{landscape}
\section{Appendix}\label{appendix}
\pagestyle{empty}

\begin{table}[h!]
\resizebox{1.75\textheight}{!}{
\begin{tabular}{|ccc|c|c|c|c|c|c|c|c|} \hline
& & &trivial&$\chi_1$ &$\chi_2$ & $\chi_3$ & $\chi_4$ & $\chi_5$ & $\chi_6$ &$\chi_7$ \\
& & & & & & $1 \leq u \leq q+1$ & $1 \leq u \leq q+1$ & $(u,v,w) \in
T$& $(q^2-q-2)/2$& $(q^2-q)/6$\\
 & & &1& $q(q-1)$ & $q^3$& $q^2-q+1$& $q(q^2-q+1)$& $(q-1)(q^2-q+1)$ & $(q+1)(q^2-q+1)$ &$(q+1)^2(q-1)$ \\
  & number & size & & & & & & & & \\  \hline
 $C_1$ & $\frac{q^2-q}{3}$ & $\frac{|G|}{q^2-q+1}$ & 1& -1& -1& 0 & 0 & 0 & 0& $B$ \\ \hline
 $C_2$ & $\frac{q^2-q}{6}$ & $\frac{|G|}{(q+1)^2}$ & 1& 2 & -1 &
  $e^{3uk}\!+\!e^{3ul}\!+\!e^{3um}$ & $-1(e^{3uk}\!+\!e^{3ul}\!+\!e^{3um} )$ & $-1 \displaystyle{\sum_{[u,v,w]}} e^{uk+vl+wm}$ & 0 & 0 \\ \hline
\end{tabular}}
\caption{Partial character Table for $\psu(3,q)$ where $3 \notdivide  q+1$ \label{tab:char1}}
\end{table}

\bigskip

\begin{table}[h!]
\resizebox{1.75\textheight}{!}{
\begin{tabular}{|ccc|c|c|c|c|c|c|c|c|c|} \hline
 & & &Trivial  &$\chi_1$ &$\chi_2$ & $\chi_3$ & $\chi_4$ & $\chi_5$ & $\chi_6$ & $\chi_7$ &$\chi_8$ \\
 & & & & & & $1 \leq u \leq (q+1)/3-1$ & $1 \leq u \leq  (q+1)/3-1$ & $k=0,1,2$ & $(u,v,w) \in T$& $(q^2-q-2)/6$&$(q^2-q-2)/18$\\ 
 & & & 1& $q(q-1)$ & $q^3$& $q^2-q+1$& $q(q^2-q+1)$ &  $(q-1)(q^2-q+1)/3$ & $(q-1)(q^2-q+1)$ & $(q+1)(q^2-q+1)$ &$(q+1)^2(q-1)$ \\
 & number & size &&& & & & & & &  \\ 
 $C_1$ & $\frac{q^2-q-2}{9}$ & $\frac{3|G|}{q^2-q+1}$ & 1& -1& -1& 0 & 0 & 0 & 0 & 0& $B$ \\ \hline 
 $C_2$ & $\frac{q^2-q-2}{18}$ & $\frac{3|G|}{(q+1)^2}$ & 1 & 2 & -1 &
 $e^{3uk}\!+\!e^{3u\ell}\!+\!e^{3um}$ & $-1(e^{3uk}\!+ \!e^{3u\ell}\!+\!e^{3um} )$ &
  $-(\omega^{k-\ell}+\omega^{\ell-k})$ & $-1 \displaystyle{\sum_{[u,v,w]}} e^{uk+vl+wm}$ & 0 & 0 \\ \hline
 $C_3$ &  $1$ & $\frac{|G|}{(q+1)^2}$ & 1 & 2 & -1 & 3 & -3 &  $-2$ or
 $1$ \footnote{This value is $-2$ if $9|q+1$ and $1$ otherwise.} & $-3(\omega^{u-v}+\omega^{v-u})$ & 0 & 0 \\ \hline
\end{tabular}
}
\caption{Partial character Table for $\psu(3,q)$ where $3 |  q+1$ \label{tab:char2}}
\end{table}

\footnotetext{This value is $-2$ if $9|q+1$ and $1$ otherwise.}
\end{landscape}

\end{document}